\documentclass{amsart}
\usepackage{amsfonts}
\usepackage{amssymb}
\usepackage{amsmath}
\usepackage{tikz}
\usepackage{hyperref}
\usepackage{dsfont}
\usepackage{graphicx}
\usepackage{mathrsfs}
\usepackage{adjustbox}

\usepackage[noend]{algpseudocode}
\usepackage{algorithm}

\newcommand{\inner}[2]{\left\langle{#1},{#2}\right\rangle}
\newcommand{\norm}[1]{\left\Vert#1\right\Vert}
\newcommand{\G}{{\rm Gr}\,}
\newcommand{\R}{\mathbb{R}}
\newcommand{\N}{\mathbb{N}}

\newtheorem{mtheorem}{Theorem}
\newtheorem{theorem}{\textbf{Theorem}}[section]

\newtheorem{definition}[theorem]{\textbf{Definition}}
\newtheorem{example}{\textbf{Example}}

\newtheorem{lemma}[theorem]{\textbf{Lemma}}
\newtheorem{proposition}[theorem]{\textbf{Proposition}}
\newtheorem{remark}[theorem]{\textbf{Remark}}

\numberwithin{equation}{section}

\newcommand{\eqdef}{\stackrel{\scriptscriptstyle\rm def}{=}}

\begin{document}

\title[Random iterations of paracontractions]{Random iterations of paracontraction maps and applications to feasibility problems}

\author[E. Matias]{Edgar Matias}
\address{Universidade Federal da Bahia, 40.170-110 Salvador, BA, Brasil}
\email{edgarmatias9271@gmail.com}

\author[M. P. Machado]{Majela Pent\' on Machado}
\address{Universidade Federal da Bahia, 40.170-110 Salvador, BA, Brasil}
\email{majela.penton@ufba.br}

\begin{abstract}
In this paper, we consider the problem of finding an almost surely common fixed point of a  family of paracontraction maps indexed on a probability space, which we refer to as the \emph{stochastic feasibility} problem. We show that a 
 random iteration of paracontraction maps driven by an ergodic stationary sequence converges, with probability one,  to a solution of the stochastic feasibility problem, provided a solution exists. As applications, we obtain \emph{non-white noise} randomized algorithms to solve the \emph{stochastic convex feasibility} problem and the problem of finding an almost surely common zero of a collection of maximal monotone operators.

\end{abstract}



\maketitle

\section{Introduction} 
\label{sc:intro}

A very common problem in several areas of mathematics is the \emph{convex feasibility} problem, which consists of finding a point belonging to the intersection of a collection $\{C_i,i\in I\}$ of convex closed subsets of $\R^k$,  where $I$ is an index set. 
It arises in numerous applications  including computerized tomography, image restoration, signal processing and many others
\cite{Combettes07,Herman,Combettes97}.

One popular way of solving the convex feasibility problem is to use \emph{projection} algorithms, originated in the works
\cite{Kaczmarz,Agmon,Motzkin,Gubin}, see also
\cite{Combettes,Bauschke}. 
If $I=\{1,\dots,m\}$, a classic version of these algorithms is the \emph{cyclic} projection algorithm, which
generates a sequence $(x_n)$ according to the recursion 
\begin{equation}
\label{eq:suc_proj}
x_{n+1}=P(x_n),
\end{equation}
where $P=P_{C_m}\circ\cdots\circ P_{C_1}$ and $P_{C_i}$ is the projection map onto the set $C_i$.
The iteration above is well suited when $m$ is not too large, however, in some  applications this may not be the case. 

To deal with the large dimensional case, one can perform the projections onto the sets $C_i$ in a random order, rather than sequentially. That is, we can replace the iteration \eqref{eq:suc_proj} by 
\begin{equation}
\label{eq:ran_proj}
X_{n+1} = P_{C_{\xi_n}}(X_n),
\end{equation}
where $(\xi_n)$ is a \emph{white noise}, i.e., $(\xi_n)$ is an  independent and identically distributed (i.i.d.) sequence of random variables, taking values from the index set $I$.
\emph{Randomized projection} algorithms driven by white noise like \eqref{eq:ran_proj} have been widely studied in recent years, see for instance
\cite{Nedic,Polyak,Strohmer} and the references therein.

Furthermore, the random nature of \eqref{eq:ran_proj} also allows to treat the case where the index set $I$ is an infinite (possibly uncountable) set. However, for this case it is more suitable to consider a stochastic reformulation of the convex feasibility problem. 
Namely,   if $\xi$  is a random variable with values in $I$, the \emph{stochastic convex feasibility}
\cite{Butnariu} problem seeks to find $x^*$ such that 
$$
x^{*}\in C_{\xi} \quad \mbox{almost surely}.
$$

It turns out that the stochastic convex feasibility problem can be formulated as finding an almost surely common fixed point for the family of projection maps $P_{C_i}$. Therefore, it can be placed in the broader context of the random fixed point theory. 

In this paper, we 
study  the \emph{stochastic feasibility} problem of finding an almost surely (a.s.) common fixed point of a family of maps.
That is, if 
 $f_{i}\colon M\to M$, $i\in I$, is a family of maps defined on a metric space $M$, the goal is to find   
$x^{*}$ such that 
\begin{equation}
\label{eq:sto_int}
f_{\xi}(x^\ast)=x^\ast\qquad\text{almost surely}.
\end{equation}
Note that for the case where $f_{i}= P_{C_{i}}$, $i\in I$, the problem above is the stochastic convex feasibility problem. 
Another particular instance of \eqref{eq:sto_int} is the problem of finding an a.s. common zero of a family of \emph{maximal monotone operators}, since finding a zero of a maximal monotone operator is equivalent to finding a fixed point of its \emph{resolvent mappings}.



The aim of our paper is to analyze an iterative algorithm for solving \eqref{eq:sto_int} in the case $\{f_{i}, i\in I\}$ is a family of paracontraction maps, see Section \ref{sbs:feasibility} for the definition of paracontraction.
Specifically, 
we consider the random iteration 
\begin{equation}
\label{eq:fr}
X_{n+1} = f_{\xi_n}(X_n),
\end{equation}
where $(\xi_n)$ is an ergodic stationary sequence of random variables taking values in $I$. If the set $C$ consisting of the points that satisfy \eqref{eq:sto_int} is non-empty and $M$ is a separable metric space such that every bounded closed set is compact, we will show that for every initial point $X_{0}=x_{0}\in M$, with probability $1$, the sequence $(X_{n})$ generated by \eqref{eq:fr} converges to some point of $C$, see Theorem \ref{th:one}.

We observe that this theorem was already obtained in 
\cite{Russell} for the case where $(\xi_n)$ is an i.i.d. sequence.
However, in the last years it has been a growing interest in the implementation of randomized algorithms in the non-independent case, since they can be more suitable for applications, for instance in \emph{distributed optimization} see 
\cite{Johansson}. Furthermore, the study of random iterations driven by a non-white noise like \eqref{eq:fr} is interesting in their own right, as they generalize the deterministic and white noise cases. 

As applications of Theorem \ref{th:one}, we will obtain a randomized projection algorithm for solving the stochastic convex feasibility problem (see Section \ref{sbs:convex}) as well as a \emph{randomized proximal point algorithm} for finding an a.s. common zero of a family of maximal monotone operators in the non-independent case (see Section \ref{sbs:proximal}). In particular, this will allow us to obtain Markovian randomized algorithms for both problems.

Finally, let us make some comments regarding the proof of Theorem \ref{th:one}. As mentioned above, this theorem was proved in
\cite{Russell} for the case $(\xi_{n})$ is an i.i.d. sequence
\cite{Russell}. We note that their proof can not be extended to the case $(\xi_{n})$ is a general stationary sequence. Indeed, in 
\cite{Russell}, the authors explore the natural Markov structure provided by $(\xi_{n})$. That is, if $(\xi_{n})$ is an i.i.d. sequence, it is known that the sequence $(X_{n})$ is a homogeneous Markov chain and hence there is a well-defined notion of \emph{stationary measures}. Using a ``loss mass'' argument, they show that every stationary measure  is supported on $C$. Then, the convergence is obtained by manipulating standard convergence results on the space of probability measures endowed with the weak star topology.  
On the other hand, when $(\xi_{n})$ is not an i.i.d. sequence, $(X_{n})$ is no longer a homogeneous Markov chain (even in the case $(\xi_{n})$ is  a Markov chain).


To overcome this situation, we follow a dynamical system approach. Indeed, it turns out that the sequence $(X_{n})$ can be seen as a random orbit of a discrete random dynamical system driven by a measure-preserving dynamical system.
 Thus, instead of exploring stationary measures, we consider \emph{invariant random measures} of the associated random dynamical system. Using the Poincar\' e recurrence theorem we obtain the recurrence of $(X_{n})$ in a certain sense, and we use this result to show that any invariant random measure is supported on $C$. From this, the convergence is obtained using convergence results on the space of random measures endowed with the \emph{narrow topology} (the weak star topology can not be used in general).

\subsection*{Organization of the paper}
In Section \ref{sc:main}, we precisely state the main definitions and results of this work. 
Section \ref{sc:ran_ppa} is devoted to the applications of Theorem \ref{th:one} to solve the stochastic convex feasibility problem and the problem of finding an a.s. common zero of a family of maximal monotone operators.
Section \ref{sc:proofs} presents some preliminaries on the space of random measures and the narrow topology, and the proofs of the results.

\section{Main results}
\label{sc:main}

\subsection{Stochastic feasibility problem}
\label{sbs:feasibility}
Let $(I,\mathscr{I},\nu)$ be a probability space and consider a collection of continuous maps $f_i:M\to M$, $i\in I$, defined on a metric space $(M,d)$. 
The problem of interest consists of finding a point $x^\ast\in M$ such that
\begin{equation}
\label{eq:sto_feas}
x^\ast\in C\eqdef\{x\in M\colon f_i(x)=x \,\mbox{ for } \nu\mbox{-almost every }i\}.
\end{equation}
We call this problem the \emph{stochastic feasibility} problem. We study the following random algorithm for solving \eqref{eq:sto_feas}:
\begin{equation}
\label{alg:1}
X_{n+1}=f_{\xi_n}(X_n)
\end{equation}
starting from an initial constant random variable  $X_0=x_0\in M$, where $(\xi_n)$ is an ergodic stationary sequence of random variables with distribution $\nu$.

Our goal is to show that the sequence $(X_k)$ generated by the random iteration \eqref{alg:1} converges to a solution of the stochastic feasibility problem, under the assumption that the maps are \emph{paracontractions}. Recall that a continuous map $f\colon M\to M$ is called a paracontraction if for every fixed point $x$ of $f$ it holds 
$$
d(f(x),y)<d(x,y)
$$
for every $y$ that is not a fixed point of $f$.

In the space $M=\R^k$, examples of paracontraction maps are the \emph{averaged mappings}. Namely, $f:\R^k\to\R^k$ is an averaged map if there is an $\alpha\in(0,1)$ such that
\begin{equation*}
\norm{f(x)-f(y)}^2 + \dfrac{1-\alpha}{\alpha}\norm{x-f(x)-(y-f(y))}^2 \leq \norm{x-y}^2, \qquad\forall x,y\in\R^k.
\end{equation*}
The projection map $P_{\tilde C}$ onto a convex closed set $\tilde{C}$ and, more generally, the \emph{resolvent mapping} associated with a \emph{maximal monotone operator} $T:\R^k\to\mathcal{P}(\R^k)$ are averaged mappings in $\R^k$ (see Section \ref{sc:ran_ppa} and Lemma \ref{lm:resolv}).

From now on, we assume that $(M,d)$ is a separable metric space such that every closed bounded subset is compact, and $f_i:M\to M$ is a paracontraction map for every $i\in I$.

\begin{mtheorem}\label{th:one}
Assume that $C$ in \eqref{eq:sto_feas} is non-empty. Then, for every initial point $X_0=x_0\in M$, the sequence $(X_n)$ of random variables generated by the random iteration \eqref{alg:1} converges almost surely to a random variable $X\in C$.
\end{mtheorem}

The theorem above was proved in \cite{Russell} in the case where $(\xi_n)$ is an i.i.d sequence. 
Therein, was also obtained the convergence of $(X_n)$ for averaged mappings on separable Hilbert spaces. 

\begin{remark}\emph{
We observe that in Theorem \ref{th:one} we can not expect convergence of the sequence $(X_n)$ if the set $C$ is empty. Indeed, in the case $(\xi_n)$ is an i.i.d. sequence of random variables, $I$ is finite and the maps $f_i$ are contractions for every $i\in I$, it is well-known that, with probability $1$, the $\omega$-\emph{limit} of $(X_n)$ is the \emph{Hutchinson attractor}, which is a non-singleton compact set if $C$ is empty, see 
\cite{fractalsevery}.}
\end{remark}

\subsection{Random iterations of paracontractions maps}
\label{ssc:ripm}

In this section, we consider a reformulation of Theorem \ref{th:one}, which will allow us to use the ergodic theory of random maps for obtaining its proof. To this end, we recall some definitions.

Let $(\Omega,\mathscr{F},\mathbb{P},\theta)$ be an \emph{ergodic measure-preserving dynamical system}, that is, $(\Omega,\mathscr{F},\mathbb{P})$ is a probability space, $\theta\colon \Omega\to \Omega$ is a measurable transformation such that $\mathbb{P}$ is an invariant measure for $\theta$, and  for every $A\in \mathscr{F}$ such that $\theta^{-1}(A)\subset A$ we have either $\mathbb{P}(A)=0$ or $\mathbb{P}(A)=1$. 

Consider a family $\{f_{\omega}\}_{\omega\in \Omega}$ of continuous maps $f_{\omega}\colon M\to M$ defined on the metric space $M$. Assume that the map $(\omega,x)\mapsto f_{\omega}(x)$ is measurable ($M$ is endowed with the Borel $\sigma$-algebra and $\Omega\times M$ with the product $\sigma$-algebra). The map $\varphi\colon \mathbb{N}\times \Omega\times M\to M$ defined by 
  \begin{equation}\label{gri}
  \varphi(n,\omega,x)\eqdef f_{\theta^{n-1}(\omega)}\circ\dots \circ f_{\omega}(x)\quad \text{for}\quad n\geq 1,\
  \end{equation}
is called a \emph{random iteration of maps} on $M$ induced by the family $\{f_{\omega}\}_{\omega\in \Omega}$ and the ergodic measure-preserving dynamical system $(\Omega,\mathscr{F},\mathbb{P},\theta)$. 

We observe that $(\theta^{n})$ is an ergodic stationary sequence on the probability space $(\Omega,\mathscr{F},\mathbb{P})$, and hence the  sequence $(\varphi(n,\omega,x))_{n\in\mathbb{N}}$ is a particular instance of the random iteration \eqref{alg:1} for $I=\Omega$ and $\xi_{n}=\theta^{n}$.

 However, for studying the convergence of the random iteration \eqref{alg:1} it is enough to study the convergence of \eqref{gri}. Indeed, first consider $(W,\mathscr{A}, \mathbf{P})$ the sample space of $(\xi_{n})$ and the measurable map $\Xi\colon W\to I^{\mathbb{N}}$ given by 
$$
\Xi(\mbox{w})=(\xi_{0}(\mbox{w}),\xi_{1}(\mbox{w}),\dots).
$$
Denote by $\mathbb{P}=\Xi_{*}\mathbf{P}$ the \emph{stochastic image} of $\mathbf{P}$ by $\Xi$, that is, $\mathbb{P}(A)=\mathbf{P}(\Xi^{-1}(A))$. Then, the natural  projections $\Pi_{n}\colon I^{\mathbb{N}}\to I$ defined by
 $$
 \Pi_{n}(\omega)=\omega_{n}, \qquad \omega=(\omega_{i})_{i\in \mathbb{N}}
 $$
constitute an ergodic stationary sequence on the probability space $(I^{\mathbb{N}},\hat{\mathscr{I}},\mathbb{P})$, where $\hat{\mathscr{I}}$ is the product $\sigma$-algebra.
Thus, the relation $\mathbb{P}=\Xi_{*}\mathbf{P}$ says that to study the convergence of the random iteration in \eqref{alg:1} it is sufficient to consider the case where $(\xi_{n})$ is taken in its canonical form, i.e.,  as the natural projections on a product space.
  
Now, we observe that the random iteration \eqref{alg:1} for the sequence $(\Pi_{n})$ can be fit into the context of \eqref{gri}. Indeed, $\mathbb{P}$ is ergodic for the shift map $\sigma\colon I^{\mathbb{N}}\to I^{\mathbb{N}}$ defined by
  $$
  \sigma((\omega_{i})_{i\in \mathbb{N}})=(\omega_{i+1})_{i\in \mathbb{N}}
  $$ 
and if we define the family of maps $\{f_{\omega}\}_{\omega\in I^{\mathbb{N}}}$ by $f_{\omega}=f_{\omega_{0}}$, then we have that the random iteration of maps $\varphi$ generated by this family and the ergodic measure-preserving dynamical system  $(I^{\mathbb{N}},\hat{\mathscr{I}}, \mathbb{P},\sigma)$ is given by 
$$
\varphi(n,\omega,x)= f_{\sigma^{n-1}(\omega)}\circ\dots \circ  f_{\omega}(x)=f_{\omega_{n-1}}\circ \dots \circ f_{\omega_{0}}(x),
$$
which is exactly the random iteration \eqref{alg:1} with $\xi_{n}=\Pi_{n}$.

We now state the reformulation of Theorem \ref{th:one}. Given a random iteration of maps $\varphi$ induced by a family of maps $\{f_{\omega}\}_{\omega\in \Omega}$ and an ergodic measure-preserving dynamical system $(\Omega,\mathscr{F},\mathbb{P},\theta)$, we define  
$$
C_{\varphi}\eqdef\{x\in M\colon \mathbb{P}(\omega\in \Omega\colon f_{\omega}(x)=x)=1\}.
$$

\begin{mtheorem}\label{paracontraction}
Consider a random iteration of paracontraction maps $\varphi$ on the metric space $M$ over the ergodic measure-preserving dynamical system $(\Omega,\mathscr{F}, \mathbb{P},\theta)$.
Assume that $C_{\varphi}\neq \emptyset$.
Then, for every $x\in M$, for $\mathbb{P}$-almost every $\omega$ the sequence $(\varphi(n,\omega,x))_{n\in \mathbb{N}}$ converges to some point of $C_{\varphi}$ (depending on $\omega$ and $x$).
\end{mtheorem}

From the discussion above, it is clear that Theorem \ref{th:one} follows from Theorem \ref{paracontraction}.

\section{Applications}
\label{sc:ran_ppa}

In this section, we present non-white noise randomized algorithms to solve the stochastic convex feasibility problem and the problem of finding an almost surely common zero of a collection of maximal monotone operators.


\subsection{Convex feasibility problem}
\label{sbs:convex}
Let $(I,\mathscr{I},\nu)$ be a probability space, and consider a family of non-empty closed convex sets $\{C_{i},i\in I\}$ in $\R^k$. The stochastic convex feasibility problem
seeks to find a point $x^{*}$ such that 
\begin{equation}
\label{eq:1.3}
x^{*}\in \{x\in \mathbb{R}^{k}\colon x\in C_{i}\,\mbox{ for }\nu\mbox{-almost every }i\}\eqdef C^\ast,
\end{equation}
or equivalently, to find $x^{*}$ such that there is a set $I_{0}$ of $\nu$-full measure satisfying 
$$
x^{*}\in \bigcap_{i\in I_{0}} C_{i}.
$$

Here, we consider the following \emph{randomized projection algorithm} for solving \eqref{eq:1.3}:
\begin{equation}
\label{eq:rit}
X_{n+1}=X_n+\alpha(P_{C_{\xi_n}}(X_n)-X_n),
\end{equation}
where $X_0=x_0\in\R^k$ is an arbitrary initial point, $(\xi_n)$ is an ergodic stationary sequence of random variables with distribution $\nu$ and $\alpha\in(0,2)$. The convergence of the algorithm above will be obtained as a consequence of Theorem \ref{th:one}. To the best of our knowledge, this is the first time that it is established convergence for randomized projection algorithms to solve \eqref{eq:1.3} in the non-independent case, generalizing previous convergence results for this type of algorithms.

In particular, we observe that in
\cite{Nedic} it was studied the random algorithm
$$X_{n+1}=X_n+\alpha_n(P_{C_{\xi_n}}(X_n)-X_n),$$
where $\alpha_n>0$. For the algorithm above, it was proved convergence to a point in $C^\ast$ assuming that $(\xi_n)$ is i.i.d and $(\alpha_n)\in\ell_2\backslash\ell_1$. If we take $\alpha_n\equiv\alpha\in(0,2)$, then the iteration above falls within the framework of \eqref{eq:rit} and, consequently, we can establish convergence relaxing the assumption on the sequence of random variables $(\xi_n)$, as well as on the sequence of stepsizes $(\alpha_n)$.


\begin{mtheorem}\label{projectionstationary}
Assume that $C^\ast$ in \eqref{eq:1.3} is non-empty. Then, for every initial point $X_0=x_0\in\R^k$, the sequence $(X_n)$ of random variables generated by \eqref{eq:rit} converges almost surely to a random variable $X\in C^\ast$.
\end{mtheorem}

As a corollary of  Theorem \ref{projectionstationary}, we obtain a \emph{Markovian randomized projection algorithm}, i.e., algorithm \eqref{eq:rit} with $(\xi_{n})$ being an ergodic stationary Markov chain.
Let us illustrate the Markovian randomized projection algorithm in the case $I$ is finite.

\begin{example}
\emph{Let $I=\{1,\dots, m\}$ and consider  an irreducible transition matrix $P=(p_{ij})$, $1\leq i,j\leq m$.  It follows from the Perron-Frobenius theorem 
that there is a unique stationary probability vector $p=(p_{1},\dots, p_{m})$. Moreover, 
$p_{i}>0$ for every $i=1,\dots, m$. Let $(\xi_{n})$ be a Markov chain with state space $I$, having transition probability matrix $P$ and initial distribution $\nu$ given by 
$$
\nu=p_{1}\delta_{1}+\dots+p_{m}\delta_{m},
$$
where $\delta_{i}$ is the Dirac measure at $i$. It is well-know that $(\xi_{n})$ is an ergodic stationary Markov chain.  Then,
it follows from Theorem \ref{projectionstationary} that the sequence $(X_{n})$ in \eqref{eq:rit} converges to a point in the set $C^\ast$ defined in \eqref{eq:1.3}.
We observe that since $I$ is the only set with $\nu$-full measure, the stochastic feasibility problem in \eqref{eq:1.3} coincides with the convex feasibility problem. Thus, the sequence $(X_{n})$ actually converges to a point in the intersection $\cap_{i\in I} C_{i}$.
}
\end{example}

\subsection{Randomized proximal point algorithm}
\label{sbs:proximal}

In this section, we consider the problem of finding an almost surely common zero of a collection of maximal monotone operators.
We observe that this problem can be seen as a stochastic feasibility problem. However, the structure of maximal monotone operators allows us to consider a randomized  algorithm that generalizes the algorithm presented in Section \ref{sbs:convex}. 
The problem of finding a common zero of a family of maximal monotone operators has been studied in several works, see for instance 
\cite{Combettes04,Kiwiel} and references therein.

We first present some preliminaries on maximal monotone operators and their properties. 
An  \emph{operator} $T$ on the space $\R^k$ is a  set-valued mapping $T:\R^k\to\mathcal{P}(\R^k)$, where $\mathcal{P}(\R^k)$ is the power set of $\R^k$.
The operator $T$ can  be  equivalently identified with its \emph{graph}, $ \G(T)\eqdef\left\{(x, v)\in \R^k\times\R^k : v \in T (x)\right\}$. 
The inverse of the operator $T$ is $T^{-1}:\R^k\to\mathcal{P}(\R^k)$ defined as $T^{-1}(v)\eqdef\{x : v\in T(x)\}$. We say that $T$ is \emph{monotone} if
\begin{equation}
\label{eq:monot}
\inner{x'-x}{v'-v}\geq 0,\quad \forall\, (x,v),(x',v')\in \G(T).
\end{equation}
Further, a monotone operator $T$ is said to be \emph{maximal} if its graph is not properly contained in the graph of any other monotone operator. 
A typical example of a maximal monotone operator is the \emph{subdifferential} $\partial f$ of a \emph{proper} lower semicontinuous convex function $f:\R^k\to(-\infty,+\infty]$.

The set $T^{-1}(0)=\{x\,:\,0\in T(x)\}$ is  called the  set  of  zeroes of the operator $T$, which is closed and convex if $T$ is maximal monotone.
A vast variety of problems such as convex or linear programming, monotone complementarity problems, variational inequalities, and equilibrium problems, can be regarded as finding a zero of a maximal monotone operator: 
\begin{equation}
\label{eq:laranja}
\text{find }\,x\in\R^k \text{ shuch that } 0\in T(x).
\end{equation}
More generally, we can consider the problem of finding an  almost surely common zero of a collection of maximal monotone operators. Namely, considering a family of maximal monotone operators $\{T_{i},i\in I\}$ defined on $\mathbb{R}^{k}$, the goal is to find $x^{*}\in \mathbb{R}^{k}$ such that
\begin{equation}\label{eq:out1}
0\in T_{i}(x^{*})\ \quad \mbox{for}\quad \nu\mbox{-almost every $i$}.
\end{equation}

A well-known technique for solving \eqref{eq:laranja} is the \emph{proximal point algorithm} 
\cite{Rockafellar}. This algorithm can be interpreted as a fixed point algorithmic framework using the \emph{resolvent} mappings associated with the operator $T$, which we briefly discuss now.

For each $u\in\R^k$ and $\lambda>0$, there is a unique $x\in\R^k$ such that
$$u\in(\mbox{Id}+\lambda T)(x),$$
where $\mbox{Id}$ is the identity mapping in $\R^k$, see 
\cite{Minty}. Hence, the function $J_{\lambda T}\eqdef(\mbox{Id}+\lambda T)^{-1}$ is single-valued and is called the \emph{resolvent} mapping of $T$ of \emph{proximal parameter} $\lambda$. 
The function $J_{\lambda T}:\R^k\to\R^k$ is non-expansive and $J_{\lambda T}(x)=x$ if and only if $0\in T(x)$. Therefore, the inclusion problem \eqref{eq:laranja} can be interpreted as a fixed point problem for $J_{\lambda T}$.
This is the motivation for the proximal point algorithm, which generates a sequence $(x_n)$ by the rule
\begin{equation*}
x_{n+1} = J_{\lambda T}(x_n),\footnote{Actually, the method in \cite{Rockafellar} allows the parameter $\lambda$ to vary with $n$.}
\end{equation*}
starting from a given point $x_0\in\R^k$. Furthermore, in the iteration above the resolvent can be replaced by a suitable affine combination of $J_{\lambda T}$ and the identity map. 
 This method is known as the \emph{generalized proximal point algorithm}
\cite{EcksteinD}. 

Inspired by this last method, we propose the following \emph{randomized proximal point algorithm} for solving problem \eqref{eq:out1}: 
\begin{equation}
\label{eq:ran_ppa}
X_{n+1} = X_n + \alpha(J_{\lambda T_{\xi_n}}(X_n)-X_n),
\end{equation}
starting from a point $X_0=x_0\in\R^k$, where $\lambda>0$ and $\alpha\in(0,2)$ are fixed, and $(\xi_n)$ is an ergodic stationary sequence of random variables. 

The following result gives the convergence of iteration \eqref{eq:ran_ppa}. Denote by $Z$ the set of all points $x^{*}$ satisfying relation \eqref{eq:out1}.
\begin{mtheorem}
\label{th:conv_ran_prox}
 Suppose that $Z$ is non-empty. Then, for every initial point $X_{0}=x_0\in\R^k$, the sequence $(X_n)$ of random variables generated by \eqref{eq:ran_ppa} converges almost surely to a random variable $X\in Z$.
\end{mtheorem}

In 
\cite{Bianchi} is analyzed (in the case of Hilbert spaces) the following similar random iteration
\begin{equation*}
X_{n+1}=J_{\lambda_n T_{\xi_n}}(X_n),
\end{equation*}
starting from an arbitrary initial point $x_0$ and assuming that $(\xi_n)$ is an i.i.d. sequence. 
Besides the fact we consider a non-independent sequence $(\xi_{n})$ in the design of the randomized proximal point algorithm, we also observe that the problem studied and the convergence results obtained in 
\cite{Bianchi} are different from those we consider here. Indeed, the author proves that almost surely the sequence $(X_n)$ converges in average to a point within the set of zeroes of the \emph{mean operator} determined by the Aumann integral of the operators, provided that $(\lambda_n)\in\ell_2\backslash\ell_1$ and the mean operator is maximal monotone.



\section{Proofs}
\label{sc:proofs}
This section is devoted to proving Theorems \ref{paracontraction}, \ref{projectionstationary} and \ref{th:conv_ran_prox}. 
Before presenting the proofs, we introduce some preliminary notions and results.

\subsection{ The space of random measures and the narrow topology}
\label{adv}

Let $(\Omega,\mathscr{F},\mathbb{P})$ be a probability space and assume in this section that $(M,d)$ a compact metric space. A probability measure $\mu$ on $\Omega\times M$ such that the first marginal of $\mu$ is $\mathbb{P}$ (i.e., $\mathbb{P}(A)=\mu(A\times M)$ for every measurable set $A$) is called  a \emph{random measure}. The space of random measures is denoted by
$\mathcal{P}_{\mathbb{P}}(\Omega\times M)$.

Following Crauel \cite{Crauel}, we consider the \emph{narrow topology} on the space of random measures, which we briefly recall now.  A \emph{random continuous function} is a map $f\colon \Omega\times M\to \mathbb{R}$ such that
\begin{itemize}
\item[(i)] for $\mathbb{P}$-almost every $\omega$, the map $x\mapsto f(\omega,x)$ is continuous; 
\item[(ii)] for every $x\in M$, the map $\omega\to f(\omega,x)$ is measurable; 
\item[(iii)] the map $\omega\mapsto \sup\{|f(\omega,x)|\colon x\in M\}$ is integrable with respect to $\mathbb{P}$. 
\end{itemize}
 
It follows from items (i)-(iii) that every random continuous function is measurable and integrable with respect to any random measure $\mu$. The smallest topology on $ \mathcal{P}_{\mathbb{P}}(\Omega\times M)$ such that the map $\mu\mapsto \mu(f)$ is continuous for every random continuous function $f$ is called the narrow topology. A sequence $\mu_{n}$ converges to $\mu$ in the narrow topology if and only if 
\begin{equation}
\label{eq:pro narrow}
\int f\, d\mu_{n}\to \int f\, d\mu
\end{equation}
for every random continuous function $f$. The space $\mathcal{P}_{\mathbb{P}}(\Omega\times M)$ endowed with the narrow topology is sequentially compact,  see Crauel \cite{Crauel}.

A random iteration of maps $\varphi$ on $M$ (recall the definition in Section \ref{ssc:ripm}) induces a map $F\colon \Omega\times M\to \Omega\times M$ defined by 
\begin{equation*}
F(\omega,x)=(\theta(\omega), f_{\omega}(x)),
\end{equation*}
which is called the \emph{skew product} induced by $\varphi$.

\begin{definition}[$\varphi$-invariant measures]
\emph{
A probability measure $\mu$ on $\Omega\times M$ is called a $\varphi$\emph{-invariant measure} if
\begin{enumerate}
\item $\mu \in  \mathcal{P}_{\mathbb{P}}(\Omega\times M)$; 
\item $\mu$ is a $F$-invariant measure.
\end{enumerate}}

\end{definition}

We observe  that the space of random measures is invariant under the action of $F$. In particular, it holds that $F^{i}_{*}\mu\in \mathcal{P}_{\mathbb{P}}(\Omega\times M)$ for every $i\geq 0$ and $\mu\in\mathcal{P}_{\mathbb{P}}(\Omega\times M)$, where $F^{i}_{*}\mu$ is the stochastic image of $\mu$ by $F^i$. 
Next, we present the following version of Krylov-Bogoliubov theorem for the narrow topology. 

\begin{proposition}[\cite{Arnold}, Theorem $1.5.8$]
\label{pr:krylov}
For every probability measure $\mu\in  \mathcal{P}_{\mathbb{P}}(\Omega\times M)$, any accumulation point of the sequence 
$$
\frac{1}{n}\sum_{i=0}^{n-1}F^{i}_{*}\mu
$$
is a $\varphi$-invariant measure.
\end{proposition}

\subsection{Recurrence Poincar\' e theorem for random iterations of maps} 
\label{Poincare}

Let $(Y, \mathscr{Y}, \eta)$ be a probability space and consider a measurable transformation $G\colon Y \to Y$ preserving the probability measure $\eta$. The classical recurrence Poincar\' e theorem says that, for every measurable set $A$, there is a subset $A_{r}\subset A$ (which may be empty if $\eta(A)=0$) with $\eta(A_{r})=\eta(A)$  such that, for every $y\in A_{r}$, we have 
$$
G^{n}(y)\in A \quad \mbox{for infinitely many}\,\, n\geq 1.
$$

A corollary of this result is the \emph{metric} recurrence Poincar\' e theorem. Namely, if the space $Y$ is a second-countable topological space, then the set of \emph{recurrent points} has $\eta$-full measure. We recall that a point $y\in Y$ is called recurrent if, for every neighborhood $V$ of $y$, there is a positive integer $n$ such that $G^{n}(y)\in V$. 

In the proof of Theorem \ref{paracontraction}, we will need the following recurrence property for random iterations.

\begin{proposition}\label{recurrentfversion}
Let $\varphi$ be a random iteration of maps on $M$. Denote by $R_{\varphi}$ the set of points $(\omega,x)$ such that, for every open neighborhood $V$ of $x$, 
$$
\varphi(n,\omega,x)\in V \quad \mbox{for infinitely many}\, \,  n.
$$   
Then, if $\mu$ is a $\varphi$-invariant measure, we have $\mu(R_{\varphi})=1$.
\end{proposition}

The proposition above could be obtained from the  metric recurrence Poincar\' e theorem applied to the skew product induced by $\varphi$, if the space $\Omega$ was a second-countable topological space. However, we are not assuming that $\Omega$ is a topological space.

\begin{proof}
Since $M$ is a separable metric space, it has a countable basis $(V_{k})_{k\in\N}$ of open sets. For each $k\in\N$, by the recurrence Poincar\' e theorem applied to the skew product $F\colon \Omega\times M\to \Omega\times M$ induced by $\varphi$ and the $F$-invariant measure $\mu$, we have that there is a subset $H_{k}\subset \Omega \times V_{k}$ such that $\mu(H_{k})=0$, and for every $(\omega,x)\in \Omega\times V_{k}\setminus  H_{k}$ the sequence $F^{n}(\omega,x)$ returns infinitely many times to $\Omega\times V_{k}$. Hence, if we take 
$$
H=\bigcup_{k=1}^{\infty} H_{k},
$$
we have that $\mu(H)=0$ and every point $(\omega,x)\in \Omega\times M\setminus H$ is in $R_{\varphi}$. 
\end{proof}

\begin{remark}
\emph{It is clear that Proposition \ref{recurrentfversion} can be proved under weaker assumptions. However, for the sake of clarity, we have preferred to present it for the context we are considering in Theorem \ref{paracontraction}.}
\end{remark}

\subsection{Proof of Theorem \ref{paracontraction}} 

Let $\varphi$ be a random iteration of paracontractions maps on $M$ and $F$ the skew product induced by $\varphi$. Recall the definition of $C_{\varphi}$ in Section \ref{ssc:ripm}. Throughout, for simplicity, we use the notation 
$$f_\omega^n(x)\eqdef\varphi(n,\omega,x).$$ 
We start by presenting the following lemmas.

\begin{lemma}\label{lm:withoutcontinuity}
Let $\{U_{k}\}_{k\in \mathbb{N}}$ be a countable basis for the topology on $M$ and denote by
 $A_{U_{k}}$ the measurable subset (possibly empty) of $\Omega$ such that, for every $(\omega,x)\in A_{U_{k}}\times U_{k}$, we have $f_{\omega}(x)\neq x$.
Then, for every $x\in M\setminus C_{\varphi}$, there is $k$ such that $x\in U_{k}$ and 
$\mathbb{P}(A_{U_{k}})>0$.
\end{lemma}

\begin{proof}
By the definition of $C_{\varphi}$, for every $x\in M\setminus C_{\varphi}$ must exist a measurable subset $A_{x}\subset\Omega$ with $\mathbb{P}(A_{x})>0$ such that, for every $\omega\in A_{x}$, it holds that
$f_{\omega}(x)\neq x$. For every $\omega\in A_{x}$,  
the set   
$$
U(\omega)\eqdef\{y\in M\colon f_{\omega}(y)\neq y \}
$$
contains $x$ and, by the continuity of $f_{\omega}$, it is open.
Therefore, we have 
$$
A_{x}=\displaystyle \bigcup_{\ell=1}^{\infty} \{\omega\in A_{x}\colon B\left( x,\frac{1}{\ell}\right)\subset U(\omega)\}.
$$
In addition, since $\mathbb{P}(A_{x})>0$, there exists $\ell$ such that 
$$
\mathbb{P}
(\{\omega\in A_{x}\colon B\left( x,\frac{1}{\ell}\right)\subset U(\omega)\})>0.
$$
Now, by the assumption that $\{U_{k}\}_{k\in \mathbb{N}}$ is a basis for the topology of $M$, there is $U_{k}\ni x$ such that $U_{k}\subset B\left( x,\frac{1}{\ell}\right)$. Then, 
$$
\{\omega\in A_{x}\colon B\left( x,\frac{1}{\ell}\right)\subset U(\omega)\}\subset A_{U_{k}},
$$
which implies that $\mathbb{P}(A_{U_{k}})>0$.
\end{proof}

\begin{lemma}\label{recurrence}
 Every  $\varphi$-invariant measure is supported on $\Omega\times C_{\varphi}$.  
\end{lemma}

\begin{proof}
Let $\mu$ be a $\varphi$-invariant measure. We need to show that $\mu(\Omega\times C_\varphi)=1$.
 If $C_{\varphi}=M$, then there is nothing to do. So we assume that $M\setminus C_{\varphi}$ is non-empty.

The Birkhoff's ergodic theorem implies that there exists a measurable set $\Lambda^{*}$ of $\mathbb{P}$-full measure such that for every $k$ with $\mathbb{P}(A_{U_{k}})>0$, we have 
$$
\theta^{n}(\omega)\in A_{U_{k}}\quad \mbox{for infinitely many}\, \,  n,
$$
 for every $\omega\in \Lambda^{*}$.
 
 Now, we observe that since $M$ is separable, there is a set $\Lambda\subset \Omega$ of $\mathbb{P}$-full measure such that
$$
f_{\omega}(c)=c 
$$
for every $c\in C_\varphi$ and every $\omega\in \Lambda$.  Taking $\Lambda'\eqdef\displaystyle\bigcap_{n=0}^{\infty} \theta^{-n}(\Lambda)$, by the definition of paracontractions, we have 
\begin{equation}\label{decrease}
d(f^{n+1}_{\omega}(x),c)\leq d(f_{\omega}^{n}(x),c)
\end{equation}
for every $c\in C_\varphi$, $(\omega,x)\in \Lambda'\times M$ and $n\in\N$.
Further, from the $\theta$-invariance of $\mathbb{P}$, it follows that $\mathbb{P}(\Lambda')=1$.

Next, we take an open set $V$ such that $V\cap C_\varphi=\emptyset$, which exists because $C_\varphi$ is closed. We claim that $\mu(\Omega\times V)=0$. Indeed, assume by contradiction that $\mu(\Omega\times V)>0$. 
Note that  
\begin{equation}
\label{eq:out2}
\mu(\Lambda^{*}\cap \Lambda'\times M)=\mathbb{P}(\Lambda^{*}\cap \Lambda')=1,
\end{equation}
where in the first equality above we use that the first marginal of $\mu$ is $\mathbb{P}$.
In particular, it follows from Proposition \ref{recurrentfversion} and equation \eqref{eq:out2} that $\mu(R_{\varphi}\cap (\Lambda'\cap \Lambda^*\times M))=1$. Since $\mu(\Omega\times V)>0$, we can find a point $(\bar \omega,\bar x)\in\Omega\times V$ such that 
$$
(\bar \omega,\bar x)\in R_{\varphi}\cap (\Lambda'\cap \Lambda^*\times V).
$$

Now, we fix $c\in C_\varphi$ and observe that $\bar \omega\in \Lambda^{*}\cap \Lambda'$ and $\bar x\notin C_{\varphi}$. By the definition of $\Lambda^{*}$ and Lemma \ref{lm:withoutcontinuity}, 
there is $k$ such that $\bar x$  is not a fixed point for the map $f_{\theta^{k}(\bar \omega)}$. Hence, since  $\bar \omega\in \Lambda'$, it follows from the definitions of $\Lambda'$ and paracontractions that 
$$
d(f^{k+1}_{\bar \omega}(\bar x),c)< d(\bar x,c),
$$
for every $n\geq k$.

Define $r\eqdef d(f^{k+1}_{\bar \omega}(\bar x),c)\geq0$ and denote by $B[c,r]$ the closed ball of radius $r$ centered at $c$. By inequality \eqref{decrease}, it follows that
$$  
d(f^{n}_{\bar \omega}(\bar x),c)\leq d(f^{k+1}_{\bar\omega}(\bar x),c)
$$
for every $n\geq k+1$, which implies 
\begin{equation}
\label{eq:tomara}
f_{\bar\omega}^{n}(\bar x)\in B[c,r]
\end{equation}
for every $n\geq k+1$. 

On the other hand, $r<  d(\bar x ,c)$ implies that there is an open neighborhood $U$ of $\bar x$, such that $U\cap B[c,r]=\emptyset$. However, the definition of $R_\varphi$ yields an integer $m\geq k+1$ such that 
$$
f_{\bar\omega}^m(\bar x)\in U,
$$
which is a contradiction with \eqref{eq:tomara}. Hence, we conclude that $\mu(\Omega\times V)=0$.

Finally, to deduce that $\mu(\Omega\times C_\varphi)=1$, we observe that the set $\Omega\times (M\setminus C_\varphi)$ is an enumerable union of sets of the form $\Omega\times V$ with $V$ open and $V\cap C_\varphi=\emptyset$. 
\end{proof}

\begin{proof}[Proof of Theorem \ref{paracontraction}]
We first assume that $M$ is a compact metric space. For every $x\in M$, we will show that 
\begin{equation}\label{dodoi3}
\lim_{n\to \infty}d(f_{\omega}^{n}(x),C_{\varphi})=0 \quad \mbox{for} \,\, \mathbb{P}\mbox{-almost every}\,\,\omega.
\end{equation}
To this end, recall the definition of $\Lambda'$ in the proof of Lemma \ref{recurrence} and that $\mathbb{P}(\Lambda')=1$. 
By equation \eqref{decrease}, we have that, for every $\omega \in \Lambda'$,
\begin{equation}\label{dodoi2}
d(f_{\omega}^{n+1}(x),C_\varphi)=\inf_{c\in C_{\varphi}}d(f_{\omega}^{n+1}(x),c)\leq \inf_{c\in C_{\varphi}}d(f_{\omega}^{n}(x),c)=d(f_{\omega}^{n}(x),C_{\varphi}),
\end{equation} 
 which means that the sequence $d(f_{\omega}^{n}(x),C_\varphi)$ is decreasing for $\mathbb{P}$-almost every $\omega$. Therefore, there is a measurable map $H\colon \Omega\to [0,\infty)$ such that   
$$
\lim_{n\to \infty}d(f_{\omega}^{n}(x),C_\varphi)=H(\omega)
$$
for $\mathbb{P}$-almost every $\omega$, and by the Dominated Convergence Theorem we have 
\begin{equation}\label{dodoi}
\lim_{n\to \infty} \int d(f_{\omega}^{n}(x),C_\varphi)\, d\mathbb{P}(\omega)=\int H(\omega)\, d\mathbb{P}(\omega).
\end{equation}

We now prove that the sequence on the left-hand side of the equation above has a subsequence converging to $0$, which will imply that $H(\omega)=0$ for $\mathbb{P}$-almost every $\omega$.
For this purpose, consider the probability measure $\mathbb{P}\times \delta_{x}\in \mathcal{P}_{\mathbb{P}}(\Omega\times M)$. 
Since $M$ is compact, the space $ \mathcal{P}_{\mathbb{P}}(\Omega\times M)$ is sequentially compact and, as a consequence, the sequence 
$$
\frac{1}{n}\sum_{i=0}^{n-1}F^{i}_{*}(\mathbb{P}\times \delta_{x})
$$
has an accumulation point $\mu$, which by Proposition \ref{pr:krylov} must be a $\varphi$-invariant measure.  Let $(n_{k})_{k\in\N}$ be a subsequence such that 
\begin{equation}\label{narrow}
\frac{1}{n_{k}}\sum_{i=0}^{n_{k}-1}F^{i}_{*}(\mathbb{P}\times \delta_{x})\to \mu
\end{equation}
in the narrow topology. Also, define the map $g\colon\Omega\times M\to \mathbb{R}$ by 
$$
g(\omega,y)=d(y,C_\varphi),
$$
which clearly is a random continuous map. From equation \eqref{narrow} and the property \eqref{eq:pro narrow} of the narrow topology applied to the map $g$, it follows
\begin{equation}
\label{eq:integral}
\lim_{k\to \infty} \int g\, d\left(\frac{1}{n_k}\sum_{i=0}^{n_k-1}F^{i}_{*}(\mathbb{P}\times \delta_{x})\right)= \int g\, d\mu.
\end{equation}

Now, let us look more closely at the integrals on both sides of the equation above.
First, we claim that $\int g\, d\mu=0$. Indeed, by Lemma \ref{recurrence} we have that 
$\mu( \Omega\times C_\varphi
)=1$ and, therefore,
\[
\begin{split}
\int g\, d\mu &=\int_{\Omega\times C_\varphi
}d(y,C_{\varphi})\, d\mu(\omega,y)=0.
\end{split}
\]
Next, for the integral on the left-hand side of \eqref{eq:integral}, we observe that 
\begin{equation*}
\int g\, d\left(\frac{1}{n_k}\sum_{i=0}^{n_k-1}F^{i}_{*}(\mathbb{P}\times \delta_{x})\right) = \dfrac{1}{n_k}\sum_{i=0}^{n_k-1}\int g\, d F^{i}_{*}(\mathbb{P}\times \delta_{x}),
\end{equation*}
and 
\[
\begin{split}
\int g\, d F^{i}_{*}(\mathbb{P}\times \delta_{x})&=\int g\circ F^{i}\, d(\mathbb{P}\times \delta_{x})\\
&=\int \int g(F^{i}(\omega,y))\, d\delta_{x}(y) \,d\mathbb{P}(\omega)\\
&=
\int g(F^{i}(\omega,x)) \,d\mathbb{P}(\omega)\\
&= \int g(\theta^{i}(\omega),f_{\omega}^{i}(x))\,d \mathbb{P}(\omega)=\int d(f_{\omega}^{i}(x), C_{\varphi})\, d\mathbb{P}(\omega),
\end{split}
\]
where the second equality above is due to Fubini's theorem.
Hence, combining the four equations above we conclude that 
\begin{equation}
\label{eq:holmes}
\begin{split}
\lim_{k\to \infty}\int \frac{1}{n_{k}}\sum_{i=0}^{n_{k}-1} d(f^{i}_{\omega}(x),C_\varphi)\,d\mathbb{P}(\omega)&=0.
\end{split}
\end{equation}

Furthermore, it follows from \eqref{dodoi2} that 
$$
d(f_\omega^{n_k}(x),C_\varphi)\leq d(f^i_\omega(x),C_\varphi)
$$
for all $0\leq i\leq n_k$ and for $\mathbb{P}$-almost every $\omega$ . Therefore, 
\[\begin{split}
\lim_{k\to \infty}\int d(f^{n_k}_{\omega}(x),C_\varphi)\,d\mathbb{P}(\omega)&=\lim_{k\to\infty} \int \frac{1}{n_{k}}\sum_{i=0}^{n_{k}-1} d(f^{n_{k}}_{\omega}(x),C_\varphi)\,d\mathbb{P}(\omega)\\
&\leq \lim_{k\to \infty}\int \frac{1}{n_{k}}\sum_{i=0}^{n_{k}-1} d(f^{i}_{\omega}(x),C_\varphi)\,d\mathbb{P}(\omega)=0,
\end{split}
\]
where the last equality follows from \eqref{eq:holmes}.
Thus, the equation above, together with \eqref{dodoi}, yields that $\int H(\omega)\,d \mathbb{P}(\omega)=0$, which implies that $H(\omega)=0$ for $\mathbb{P}$-almost every $\omega$.

To conclude the proof of the theorem in the case $M$ is compact, let $\Lambda''$ be the set of $\omega$ such that   the sequence $d(f_{\omega}^{n}(x),C_{\varphi})$ is decreasing and converges to $0$. Fix $\omega\in \Lambda''$ and let $c$ be any accumulation point of $f_{\omega}^{n}(x)$. By the choice of $\omega$, it is clear that $c\in C_\varphi$. Since $d(f^n_{\omega}(x), c)$ is decreasing, we conclude that $f^{n}_{\omega}(x)$ actually converges to $c\in C_\varphi$. Finally, combing \eqref{dodoi3} with \eqref{dodoi2}, we obtain that $\mathbb{P}(\Lambda'')=1$.

Now, we assume that $M$ is separable and every bounded closed subset is compact. 
Fix $c\in C_\varphi$, let $r>0$ and consider the compact metric space $\hat M=B[c,r]$. For every $\omega\in \Lambda'$ and every $y\in \hat M$,  by the definition of paracontraction we have $f_{\omega}(y)\in \hat M$. Thus, this allows us to consider the family of paracontractions $\{\hat f_\omega\}_{\omega\in\Lambda'}$, given by $\hat{f}_\omega:\hat M\to \hat M$, $\hat{f}_\omega(y)=f_\omega(y)$. Also, we consider the ergodic measure-preserving dynamical system $(\Lambda',\mathscr{F}',\mathbb{P}',\hat\theta)$ where 
$$
\mathscr{F}'=\{A\cap \Omega\colon A\in \mathscr{F}\},\quad \mathbb{P}'=\mathbb{P}_{|\mathscr{F}'} \quad \mbox{and} \quad \hat\theta\colon \Lambda'\to \Lambda',\,\hat{\theta}(\omega)=\theta(\omega).
$$
We note that $\mathbb{P}'$ is $\hat\theta$-invariant since $\mathbb{P}(\Lambda')=1$. Hence, it follows from the compact case applied to the random iteration of maps induced by the family $\{\hat f_{\omega}\}_{\omega\in\Lambda'}$ and $(\Lambda',\mathscr{F}',\mathbb{P}',\hat\theta)$ that, for every $x\in \hat M$, the sequence $\hat \varphi(n,\omega,x)$ converges to some point of $C_{\hat \varphi}\subset C_{\varphi}$ for $\mathbb{P}'$-almost every $\omega$. In particular, for every $x\in \hat M$, 
the sequence $\varphi(n,\omega,x)$ converges to some point of $C_{\varphi}$ for $\mathbb{P}$-almost every $\omega$. 

To conclude the proof of the theorem, we observe that every point of $x\in M$ belongs to some closed ball centered at $c$.

\end{proof}

\subsection{Proofs of Theorems \ref{projectionstationary} and \ref{th:conv_ran_prox}}

We note that Theorem \ref{projectionstationary} can be placed in the context of Theorem \ref{th:conv_ran_prox}. Thus, we first prove Theorem \ref{th:conv_ran_prox}, and  Theorem \ref{projectionstationary} will be obtained as a consequence. Nevertheless, we observe that Theorem \ref{projectionstationary} could be proved directly.

\begin{proof}[Proof of Theorem \ref{th:conv_ran_prox}]
We define $Q_{i}(x)\eqdef x + \alpha(J_{\lambda T_i}(x)-x)$ and observe that the sequence $(X_n)$ in \eqref{eq:ran_ppa} can be rewritten as
$$X_{n+1}=Q_{\xi_n}(X_n),$$
which is a particular instance of the random iteration \eqref{alg:1}. 
We also observe that 
\begin{equation*}
Q_{i}(x)=x \quad \mbox{if and only if}\quad  0\in T_i(x).
\end{equation*}
Therefore, proving that $(X_n)$ converges to a point of $Z$ is equivalent to proving that it converges to a point in the set $C$ given in \eqref{eq:sto_feas} with $f_i=Q_i$.

Hence, if we prove that $Q_i$ is a paracontraction, the theorem would be a consequence of Theorem \ref{th:one}, since $\R^k$ is a separable metric space such that every bounded closed subset is compact. The next lemma shows that $Q_i$ is an averaged map and, in particular, a paracontraction map. This is a classical result that we present here for the sake of completeness.

\begin{lemma}
\label{lm:resolv}
Let $T$ be a maximal monotone operator, $\lambda>0$, $\alpha\in(0,2)$ and consider the map $Q:\R^k\to\R^k$, defined as
\begin{equation*}
Q(x)=x+\alpha(J_{\lambda T}(x)-x).
\end{equation*}
Then, $Q$ is an averaged mapping with parameter $\alpha/2$.
\end{lemma}

\begin{proof}
First, we consider the \emph{Yosida approximation} of $T$, the operator $A_{\lambda T}\eqdef(I-J_{\lambda T})/\lambda$, and observe that for all $x\in\R^k$ we have the following

\item[(i)] $A_{\lambda T}(x)\in T(J_{\lambda T}(x))$;
\item[(ii)] $x=J_{\lambda T}(x)+\lambda A_{\lambda T}(x)$;
\item[(iii)] $Q(x)=x-\alpha\lambda A_{\lambda T}$.

Therefore, by (iii), for any $x,y\in\R^k$ it holds 
\begin{equation*}
\begin{split}
\norm{Q(x)-Q(y)}^2 = & \norm{x-\alpha\lambda A_{\lambda T}(x)-(y-\alpha\lambda A_{\lambda T}(y))}^2\\
= & \norm{x-y}^2 - 2\alpha\lambda\inner{x-y}{A_{\lambda T}(x)-A_{\lambda T}(y)} \\
& + \alpha^2\lambda^2\norm{A_{\lambda T}(x)-A_{\lambda T}(y)}^2\\
= & \norm{x-y}^2 - 2\alpha\lambda\inner{J_{\lambda T}(x)-J_{\lambda T}(y)}{A_{\lambda T}(x)-A_{\lambda T}(y)}\\
& -2\alpha\lambda^2\norm{A_{\lambda T}(x)-A_{\lambda T}(y)}^2 + \alpha^2\lambda^2\norm{A_{\lambda T}(x)-A_{\lambda T}(y)}^2,
\end{split}
\end{equation*}
where the last equality above follows from (ii) and a simple manipulation. Now, combining (i) with the monotonicity property \eqref{eq:monot} and equation above, we obtain
\begin{equation*}
\begin{split}
\norm{Q(x)-Q(y)}^2 \leq & \norm{x-y}^2 - \alpha(2-\alpha)\lambda^2\norm{A_{\lambda T}(x)-A_{\lambda T}(y)}^2\\
= & \norm{x-y}^2 - \alpha(2-\alpha)\norm{\lambda A_{\lambda T}(x)-\lambda A_{\lambda T}(y)}^2.
\end{split}
\end{equation*}
Substituting (ii) in the second term of the right-hand side of the above equation we have
\begin{equation*}
\norm{Q(x)-Q(y)}^2 \leq \norm{x-y}^2 - \alpha(2-\alpha)\norm{x -J_{\lambda T}(x)-(y-J_{\lambda T}(y))}^2,
\end{equation*}
and, after simple manipulations, we conclude.
\end{proof}
This completes the proof of the theorem.
\end{proof}

\begin{proof}[Proof of Theorem \ref{projectionstationary}]
Consider the \emph{normal cone} $N_{C_i}$ of the non-empty closed convex set $C_i$, which is defined as 
\begin{equation*}
N_{C_i}(x)\eqdef\left\lbrace\begin{array}{ll}
\{v\in\R^k\,:\,\inner{v}{x'-x}\leq0\,\,\,\forall x'\in C_i\}\quad & \text{if }x\in C_i\\
\emptyset & \text{otherwwise}
\end{array} \right..
\end{equation*}
It is well-known that $N_{C_i}$ is a maximal monotone operator and that finding a point in $C_i$ is equivalent to finding a point in $N_{C_i}^{-1}(0)$. 
Therefore, the stochastic convex feasibility problem is a special case of the problem \eqref{eq:out1}.

Hence, the theorem follows from Theorem \ref{th:conv_ran_prox} noting that $P_{C_i}=(\mbox{Id}+N_{C_i})^{-1}$.

\end{proof}

\end{document}